\documentclass[a4paper]{amsart}

\usepackage[T1]{fontenc}
\usepackage[utf8]{inputenc}

\usepackage{amssymb}
\usepackage{amsthm}
\usepackage{amsmath}
\usepackage{amscd}
\usepackage{mathtools}
\usepackage{verbatim}
\usepackage{tikz-cd}
\usepackage{color}
\usepackage{graphicx}
\usepackage{import}

\usepackage[foot]{amsaddr}

\newtheorem{theorem}{Theorem}[section]
\newtheorem{thmx}{Theorem}

\newtheorem{proposition}[theorem]{Proposition}
\newtheorem{lemma}[theorem]{Lemma}

\theoremstyle{remark}

\theoremstyle{remark}
\newtheorem*{notation}{Notation}

\renewcommand{\P}{\mathbb{P}}
\renewcommand{\O}{\mathcal{O}_C}
\newcommand{\OO}{\mathcal{O}}
\newcommand{\SU}{\mathcal{SU}_C(2)}
\newcommand{\SUr}{\mathcal{SU}_C(r)}
\newcommand{\SUw}{\mathcal{SU}_{C_w}(2)}
\newcommand{\genSU}{\mathcal{SU}^{gs}_C(2)}
\newcommand{\PB}{\P_B^{3g-6}}
\newcommand{\PD}{\P_D^{3g-2}}
\newcommand{\PN}{\P_N^{2g-2}}
\newcommand{\PcuatroN}{\P_N^{4}}
\newcommand{\Kum}{\operatorname{Kum}}
\newcommand{\Ext}{\operatorname{Ext}}
\newcommand{\bl}{\operatorname{bl}}

\newcommand{\M}{\mathcal{M}}
\newcommand{\N}{\mathcal{N}}
\renewcommand{\L}{\mathcal{L}}
\newcommand{\I}{\mathcal{I}}

\newcommand{\Sec}{\operatorname{Sec}}
\newcommand{\Hom}{\operatorname{Hom}}
\newcommand{\Zeroes}{\operatorname{Zeroes}}
\newcommand{\Pic}{\operatorname{Pic}}
\newcommand{\Jac}{\operatorname{Jac}}
\newcommand{\Sing}{\operatorname{Sing}}

\newcommand{\MGIT}{\M_{0,2g}^{\operatorname{GIT}}}
\newcommand{\MGITsix}{\M_{0,6}^{\operatorname{GIT}}}

\newcommand{\tto}{\dashrightarrow}
\newcommand{\p}{\varphi_D}
\newcommand{\pL}{\varphi_L}
\newcommand{\pPN}{\varphi_{D,N}}
\newcommand{\pPcuatroN}{\varphi_D|_{\PcuatroN}}
\newcommand{\pPB}{\varphi_D(\PB)}
\newcommand{\pproj}{p_{\P_c}}

\makeatletter
\renewcommand{\paragraph}{%
      \@startsection{paragraph}{4}%
      {\z@}{1ex \@plus 1ex \@minus .2ex}{-1em}%
      {\normalfont\normalsize\bfseries}%
}
\makeatother

\title{Involutions on moduli spaces of vector bundles \\ and GIT quotients}
\author{Néstor Fernández Vargas}
\address{Univ Rennes, CNRS, IRMAR - UMR 6625, F-35000 Rennes, France}
\email{nestor.fernandez-vargas@univ-rennes1.fr}
\thanks{The author gratefully acknowledges support by the Centre Henri Lebesgue (ANR-11-LABX-0020-01) and by the IMAG, Univ Montpellier, CNRS, Montpellier, France}


\begin{document}

\begin{abstract}
  Let $C$ be a hyperelliptic curve of genus $g \geq 3$.
  We give a new description of the theta map for moduli spaces of rank 2 semistable vector bundles with trivial determinant.
  In orther to do this, we describe a fibration of (a birational model of) the moduli space, whose fibers are GIT quotients $(\mathbb{P}^1)^{2g}//\operatorname{PGL(2)}$. Then, we use recent results of Kumar to identify the restriction of the theta map to these GIT quotients with some explicit osculating projection. 
  As a corollary of this construction, we obtain a birational equivalence between the ramification locus of the theta map and a fibration in Kummer $(g-1)$-varieties over $\mathbb{P}^g$.
\end{abstract}

\maketitle

{}\section{Introduction} \label{sec:introduction}
  Let $C$ be a complex smooth curve of genus $g \geq 3$.
  Let $\SUr$ be the moduli space of semistable vector bundles of rank $r$ with trivial determinant on $C$. This moduli space is a normal, projective, unirational variety of dimension $(r^2 - 1)(g - 1)$. The study of the projective structure of the moduli spaces of vector bundles in low rank and genus has produced some beautiful descriptions, frequently meeting constructions issued in the context of classical geometry. 

  For example, in the case of a hyperelliptic curve $C$, Desale and Ramanan \cite{desale_ramanan} characterize the quotient $\SU/i^*$ of the moduli space of rank 2 vector bundles by the map $i^*$ induced by the hyperelliptic involution $i$.
  They show that there exists two quadrics $Q_1$ and $Q_2$ in the $(2g + 1)$-dimensional projective space such that the quotient $\SU/i^*$ is isomorphic to the variety of $g$-dimensional linear subspaces contained in $Q_1$, belonging to a fixed system of maximal isotropic spaces, and intersecting $Q_2$ in quadrics of rank $\leq 4$.
  Some other beautiful results regarding the projective structure of $\SUr$ can be found in \cite{pauly_coble} and \cite{ortega_coble}.

  The natural map $\alpha_{\L}: \SUr \dashrightarrow |\L|^*$ induced by the determinant line bundle $\L$ on $\SUr$ is determined by the $r \Theta$ linear series on the Jacobian variety $\Jac(C)$. More precisely, let $\Pic^{g-1}(C)$ be the Picard variety of divisors of degree $g - 1$ over $C$. For every $E \in \SUr$, let us define
$$\theta(E) := \{ L \in \Pic^{g-1}(C) \ | \ h^0(C, E \otimes L) \not = 0 \}. $$
  If $\theta(E)$ is not equal to $\Pic^{g-1}(C)$, we have that $\theta(E)$ is a divisor in $\Pic^{g-1}(C)$ lying in the linear system $|r \Theta|$, where $\Theta$ is the canonical divisor in  $\Pic^{g-1}(C)$.
  This way, we obtain a rational map
\begin{align*}
  \theta : \SUr \dashrightarrow |r \Theta|
\end{align*} 
which is canonically identified to $\alpha_{\L}$ \cite{beauville_narasimhan_ramanan}.

  Let us now fix $r = 2$. In this setting, the map $\theta$ is a finite morphism \cite{raynaud}. When $g = 2$, the map $\theta$ is an isomorphism onto $\P^3$ \cite{narasimhan_ramanan_moduli}.
  For $g \geq 3$, the map $\theta$ is an embedding if $C$ is non-hyperelliptic, and it is a 2:1 map if $C$ is hyperelliptic \cite{desale_ramanan, beauville_rang2, brivio_verra, vangeemen_izadi} (see Section \ref{sec:moduli_vector_bundles_theta} for more details). 
  If $g \geq 3$, the singular locus $\Sing(\SU)$ is the locus of decomposable bundles $L \oplus L^{-1}$, with $L \in \Jac(C)$. The map $\Jac(C) \to \SU$ defined by $L \to L \oplus L^{-1}$ identifies the Kummer variety of $\Jac(C)$ with the singular locus of $\SU$.

  The goal of this paper is to describe the geometry associated to the map $\theta$ in the case $r = 2$ and $C$ hyperelliptic.
  In the non-hyperelliptic case, the paper \cite{alzati_bolognesi} outlines a connection between the moduli space $\SU$ and the moduli space $\M_{0,n}$ of rational curves with $n$ marked points. A generalization of \cite{alzati_bolognesi} has been given in \cite{bolognesi_brivio}.
  In the present work, the link with the moduli space of curves offers also a new description of the $\theta$-map if $C$ is hyperelliptic.

  Let $C$ be a hyperelliptic curve of genus $g \geq 3$. Let $D$ be an effective divisor of degree $g$ on $C$.
  The first result of the present work is an extension of \cite[Theorem 1.1]{alzati_bolognesi} to the hyperelliptic setting:
\begin{proposition} \label{prop:fibration_in_M}
  There exists a fibration $p_D:\SU \dashrightarrow |2D| \cong \P^g$ whose general fiber is birational to $\M_{0,2g}$. Moreover, we have:
  \begin{enumerate}
    \item For every generic divisor $N \in |2D|$, there exists a $2g$-pointed projective space $\PN$ and a rational dominant map $h_N: \PN \dashrightarrow p_D^{-1}(N)$ such that the fibers of $h_N$ are rational normal curves passing by the $2g$ marked points.
    \item The family of rational normal curves defined by $h_N$ is the universal family of rational curves over the generic fiber $\M_{0,2g}$.
  \end{enumerate}
\end{proposition}
  
  The $2g$-pointed space $\PN$ appears naturally as a classifying space for certain extension classes. More precisely, consider extensions
\begin{align*}
(e) \quad  0 \to \O(-D) \to E_e \to \O(D) \to 0.
\end{align*}
  These are classified by the projective space $$\PD := \P \Ext^1(\O(D), \O(-D)) = |K + 2D|^*,$$ where $K$ is the canonical divisor on $C$.
  Since the divisor $K + 2D$ is very ample, the linear system $|K + 2D|$ embeds the curve $C$ in $\PD$. The projective space $\PN$ is defined as the span in $\PD$ of the $2g$ marked points $p_1, \ldots, p_{2g}$ on $C$ defined by the effective divisor $N \in |2D|$.

  Our aim is to describe the map $\theta$ restricted to the generic fibers of the fibration $p_D$. To this end, the following construction is crucial:

  Let $p, i(p)$ be two involution-conjugate points in $C$; and consider the line $l \subset \PD$ secant to $C$ and passing through $p$ and $i(p)$. We show that this line intersects the subspace $\PN$ in a point. Moreover, the locus $\Gamma \subset \PN$ of these intersections as we vary $p, i(p)$ is a rational normal curve passing by the points $p_1, \ldots, p_{2g}$.
  It follows from Proposition \ref{prop:fibration_in_M}, the map $h_N$ contracts the curve $\Gamma$ onto a point $P \in p_D^{-1}(N) \cong \M_{0,2g}$.

  In the article \cite{kumar}, Kumar defines the linear system $\Omega$ of $(g-1)$-forms on $\P^{2g - 3}$ vanishing with multiplicity $g-2$ at $2g - 1$ general points. He shows that $\Gamma$ induces a birational map $i_{\Omega}: \P^{2g - 3} \dashrightarrow \MGIT$ onto the GIT quotient of the moduli space $\M_{0,2g}$. The partial linear system $\Lambda \subset \Omega$ of forms vanishing with multiplicity $g - 2$ at an additional general point $e \in \P^{2g - 3}$ induces a rational projection $\kappa : \MGIT \dashrightarrow | \Lambda |^*$. In particular, $\kappa$ is an osculating projection centered on the point $w = i_{\Omega}(e)$. Kumar also shows that the map $\kappa$ is 2:1. We describe birationally the restrictions of $\theta$ to the fibers $p_D^{-1}(N)$ using Kumar's map: 
\begin{thmx} \label{th:theta_is_kappa}
  The map $\theta$ restricted to the fibers $p_D^{-1}(N)$ is Kumar's osculating projection $\kappa$ centered at the point $P = h_N(\Gamma)$, up to composition with a birational map.
\end{thmx}

  Furthermore, Kumar shows that the image of $\kappa$ is a connected component of the moduli space $\SUw$, where $C_w$ is the hyperelliptic 2:1 cover of $\P^1$ ramifying over the $2g$ points defined by $w$. He also proves that the ramification locus of the map $\kappa$ is the Kummer variety $\Sing(\SU) = \Kum(C_w) \subset \SUw$.
  These results due to Kumar combined with Theorem \ref{th:theta_is_kappa} allow us to describe the ramification locus of the map $\theta$:
\begin{thmx}
  The ramification locus of the map $\theta$ is birational to a fibration over $|2D| \cong \P^g$ in Kummer varieties of dimension $g-1$.
\end{thmx}

\medskip

  The fundamental tools in our arguments are the classification maps 
\begin{align*}
  f_L : \P \Ext^1(L^{-1},L) \dashrightarrow \SU,
\end{align*}
where $L$ is a line bundle (in this paper, we will use $L = \O(-D)$). More precisely, the map $f_L$ associates to the equivalence class of a extension $(e)$ the vector bundle $E_e \in \SU$ sitting in the sequence. Consequently, the base locus of $f_L$ is the locus of unstable classes in $\P \Ext^1(L^{-1},L)$. 

  The use of classification maps constitutes a classical approach to the study of moduli spaces of vector bundes. For example, they have been used by Atiyah \cite{atiyahvectorbundles} to study vector bundles over elliptic curves, and by Newstead \cite{newstead_stable} to study the moduli space of rank 2 semistable vector bundles with odd determinant in the case $g = 2$. 
  In \cite{alzati_bolognesi} and \cite{bolognesi_conicbundle}, they have been used to study the moduli space $\SU$ when $C$ is a curve of genus $g \geq 2$, non-hyperelliptic if $g > 2$.

  We conclude by giving account of the situation in low genus. Let $\pL := \theta \circ f_L$. A precise examination of the base locus of the restriction map $\pL|_{\PN}$ leads to the following result:
\begin{thmx}
  Let $C$ be a hyperelliptic curve of genus 3, 4 or 5. Then, for generic $N$, the restriction of $\pL$ to the subspace $\PN$ is exactly the composition $\kappa \circ h_N$.
\end{thmx}

  \begin{notation} Let $\P^n = \P(\mathbb{C}^{n+1})$ denote the $n$-dimensional complex projective space.
  Throughout this paper, a form $F$ of degree $r$ on $\P^n$ will denote element of the vector space $H^0(\P^n, \mathcal{O}_{\P^n}(r)) = \operatorname{Sym}^r(\mathbb{C}^{n+1})^*$. 
  If we fix a basis $x_0, \ldots, x_n$ of $(\mathbb{C}^{n+1})^*$, $F$ is simply a homogeneous polynomial of degree $r$ on $x_0, \ldots, x_n$.
\end{notation}
\paragraph{Acknowledgements}
  This work is part of my PhD thesis. I would like to deeply thank my PhD advisors Michele Bolognesi and Frank Loray for their great support and guidance.

{}\section{Moduli of vector bundles} \label{sec:moduli}
  
  We briefly recall here some results about moduli of vector bundles. For a more detailed reference, see \cite{beauville_vectorbundles}.

\subsection{Moduli of vector bundles and the map $\theta$} \label{sec:moduli_vector_bundles_theta}

  Let $C$ be a smooth genus $g$ algebraic curve with $g \geq 2$. 
  Let us denote by $\Pic^d(C)$ the Picard variety of degree $d$ line bundles on $C$. The Jacobian of C is $\Jac(C) = \Pic^0(C)$. The canonical divisor $\Theta \subset \Pic^{g-1}(C)$ is defined set-theoretically as
$$\Theta := \{ L \in \Pic^{g-1}(C) \ | \ h^0(C,L) \not = 0 \}.$$
  Let $\SU$ be the moduli space of semistable rank 2 vector bundles on $C$ with trivial determinant. This variety parametrizes S-equivalence classes of such vector bundles, where the S-equivalence relation is defined as follows: every semistable vector bundle $E$ admits a \emph{Jordan-Hölder filtration}
$$ 0 = E_0 \subsetneq E_1 \subsetneq E_2 = E$$
such that the quotients $E_1 = E_1/E_0$ and $E_2/E_1$ are stable of slope equal to the slope of $E$. The vector bundle  $$\mbox{gr}(E) := E_1 \oplus (E_2/E_1)$$ is the \emph{graded bundle} associated to $E$. By definition, two semistable vector bundles $E$ and $E'$ on $C$ are S-equivalent if $\mbox{gr} (E) \cong \mbox{gr} (E')$. In particular, two stable bundles are S-equivalent if and only if they are isomorphic.

  The Picard group $\Pic(\SU)$ is isomorphic to $\mathbb{Z}$, and it is generated by a line bundle $\mathcal{L}$ called the \emph{determinant bundle} \cite{drezet_narasimhan}. For every $E \in \SU$, let us define the \emph{theta divisor}
$$\theta(E) := \{ L \in \Pic^{g-1}(C) \ | \ h^0(C, E \otimes L) \not = 0 \}. $$

  In the rank 2 case, $\theta(E)$ is a divisor in the linear system $|2 \Theta| \cong \P^{2^g - 1}$ (this is not true in general in higher rank), which leads to the definition of the theta map $$\theta : \SU \longrightarrow |2 \Theta|.$$

  If $\theta$ is a morphism, then we have that  $\theta$ is finite. 
  Indeed, since the linear system $|\L|$ is ample, the map $\theta$ cannot contract any curve.
  It is known that the map $\theta$ is a morphism if $r = 2$; $r = 3$ and $g = 2$ or $3$; or $r = 3$ and $C$ is generic \cite{raynaud}.
  The map $\theta$ is not a morphism if $r \gg 0$ \cite{raynaud,casalaina-martin_gwena_teixidor}, and it is generically injective for $C$ general and $g \gg r$ \cite{brivio_verra_plucker}.

  Let us now fix $r = 2$. In particular, the map $\theta$ is a finite morphism. If $C$ is not hyperelliptic, $\theta$ is known to be an embedding \cite{brivio_verra, vangeemen_izadi}. This is also the case in genus 2, where $\theta$ is actually an isomorphism onto $\P^3$ \cite{narasimhan_ramanan_moduli}. If $C$ is hyperelliptic of genus $g \geq 3$, we have that $\theta$ factors through the involution $$E \mapsto i^*E$$ induced by the hyperelliptic involution $i$, embedding the quotient $\SU/i^*$ into $|2 \Theta|$ \cite{desale_ramanan, beauville_rang2}. 

\subsection{The classifying maps} \label{sec:classifying_maps}
  Let $D$ be a general degree $g$ effective divisor on $C$. Let us consider isomorphism classes of extensions
\begin{align*}
(e) \quad  0 \to \O(-D) \to E_e \to \O(D) \to 0.
\end{align*}
  These extensions are classified by the $(3g - 2)$-dimensional projective space
\begin{align*}
  \PD := \P \Ext^1(\O(D), \O(-D)) = |K + 2D|^*,
\end{align*}
where $K$ is the canonical divisor of $C$.
  The divisor $K + 2D$ is very ample and embeds $C$ as a degree $4g - 2$ curve in $\PD$. 
  Let us define the rational surjective \emph{extension map}
\begin{align*}
  f_D : \PD  &\tto \SU 
\end{align*}
which sends the extension class $(e)$ to the vector bundle $E_e$. The composition map 
\begin{align*}
  \p : = \theta \circ f_D : \PD \tto | 2 \Theta |
\end{align*}
has been described by Bertram in \cite{bertram}. 
  More precisely, Theorem 2 of \cite{bertram} gives an isomorphism
\begin{align} 
  H^0(\SU, \mathcal{L}) \cong H^0(\PD, \I_C^{g-1}(g)),
\end{align}
where $\I_C$ is the ideal sheaf of $C$. 
  Consequently, we have the following characterization:
\begin{theorem}[Bertram \cite{bertram}] \label{th:bertram}
  The map $\p$ is given by the linear system $|\I_C^{g-1}(g)|$ of forms of degree $g$ vanishing with multiplicity at least $g - 1$ on $C$. 
\end{theorem} 
  Let us denote by $\Sec^n(C)$ the variety of $(n + 1)$-secant $n$-planes on $C$.
  We have that the singular locus of $\Sec^{n+1}(C)$ is the secant variety $\Sec^{n}(C)$ for every $n$.
  The linear system $|\I_C^{g-1}(g)|$ is characterized as follows:

\begin{proposition}[\protect{\cite[Lemma 2.5]{alzati_bolognesi}}] \label{prop:linear_systems_coincide}
  The linear system $|\I_C^{g-1}(g)|$ and $|\I_{\Sec^{g-2} (C)}(g)|$ on $\PD$ are the same.
\end{proposition}
\begin{proof}
  We reproduce here the proof for the reader's convenience.
  The elements of both linear systems can be seen as symmetric $g$-linear forms on the vector space $H^0(C, K + 2D)^*$. 
  Let $F$, $G$ be such forms. 
  Then, $F$ belongs to $|\I_C^{g-1}(g)|$ (resp. $G$ belongs to $|\I_{\Sec^{g-2} (C)}(g)|$) if and only if
\begin{align*}
  F(p_1, \ldots, p_g) &= 0 \quad  \text{ for all } p_k \in C \text{ such that } p_i = p_j \text{ for some } 1 \leq i,j \leq g \\
  G(p, \ldots, p) &= 0 \quad \text{ for any linear combination } p = \sum_{k=1}^{g-1} \lambda_i p_i \text{, where } p_i \in C .
\end{align*}
  One can show that these conditions are equivalent by exhibiting appropriate choices of $\lambda_i$. 
\end{proof}

\subsection{The exceptional fibers of $f_D$}
  Since $\dim \SU = 3g - 3$, the generic fiber of $f_D$ has dimension one. The set of stable bundles for which $\dim(f_D^{-1}(E)) > 1$ is a proper subset of $\SU$. 
  In order to study this subset, and following \cite{alzati_bolognesi}, we introduce the "Serre dual" divisor
\begin{align*}
  B := K - D
\end{align*} 
  with $\deg(B) = g-2$. As in the previous paragraphs, the isomorphism classes of extensions
\begin{align*}
  \quad  0 \to \O(-B) \to E \to \O(B) \to 0
\end{align*}
 are classified by the projective space $$\PB := \P \Ext^1 (\O(B), \O(-B)) = |K + 2B|^*.$$ We also have the rational classifying map $f_B :  \PB  \tto \SU$ defined in the same way as $f_D$.
   
\begin{proposition} \label{prop:exceptional_fibers}
  Let $E \in \SU$ be a stable bundle. Then 
\begin{align*}
  \dim(f_D^{-1}(E)) \geq 2 \qquad  \text{if and only if} \qquad   E \in \overline{f_B(\PB)}.
\end{align*}
\end{proposition}
\begin{proof}
  Let $E$ be a stable bundle. Then, by Riemann-Roch and Serre duality theorems, the dimension of $f^{-1}_D(E)$ is given by
\begin{align*}
   h^0(C, E \otimes \O(D)) &= h^0(C, E \otimes \O(B)) + 2g - 2(g-1) 
\\ &= h^0(C, E \otimes \O(B)) + 1
\end{align*}
Thus, $\dim(f_D^{-1}(E)) > 2$ if and only if there exists a non-zero sheaf morphism $\O(-B) \to E$. This is equivalent to $E \in \overline{f_B(\PB)}$.
\end{proof}

  If $g > 2$, the divisor $|K + 2B|$ embeds $C$ as a degree $4g - 6$ curve in $\PB$ (recall that $\P \Ext^1 (\O(B),\O(-B)) = |K + 2 B|^*$). 
  Again by Theorem \ref{th:bertram}, the map $\varphi_B$ is given by the linear system $|\I_C^{g-3}(g-2)|$.
  Moreover, Pareschi and Popa \cite[Theorem~4.1]{pareschi_popa} proved that this linear system has projective dimension $\left( \sum_{i = 0}^{g-2} {\binom{g}{i}} \right) - 1$ .

  Let us denote by $\P_c$ the linear span of $\theta(\overline{f_B(\PB)})$ in $|2 \Theta|$.  
  Since the map $\theta$ is finite, $\P_c$ has projective dimension $\left[ \sum_{i = 0}^{g-2} {\binom{g}{i}} \right] - 1$, and Proposition \ref{prop:exceptional_fibers} also applies to $\p$: the fibers of $\p$ with dimension $\geq 2$ are those over $\P_c$.

{}\section{The projection in $|2 \Theta|$}
  In this Section, we describe the projection with center $\P_c$, seen as a linear subspace of $|2 \Theta|$.

  Let $\pproj$ be the linear projection in $|2 \Theta|$ with center $\P_c$. Recall that $\dim \P_c = \left[ \sum_{i = 0}^{g-2} \binom{g}{i} \right] - 1$. We can check that the supplementary linear subspaces of $\P_c$ in $|2 \Theta|$ are of projective dimension $g$. Thus, the image of $\pproj$ is a $g$-dimensional projective space.
  Let us write $$\genSU := \SU \setminus (\Kum(C) \cup \overline{\pPB})$$ where \emph{gs} stands for \emph{general stable} bundles.
  Recall that the space $H^0(C, E \otimes \O(D))$ has dimension 2 for $E \in \genSU$. Consequently, we can pick two sections $s_1$ and $s_2$ that constitute a basis for this space.

\begin{theorem} \label{th:proyeccion}
  The image of the projection $\pproj$ can be identified with the linear system $|2 D|$ on $C$, in such way that the restriction of the projection $\pproj$  to $\theta(\genSU)$ coincides with the map
\begin{align*}
  \theta(\genSU) \to |2D|  \\
  \theta(E) \mapsto \Zeroes(s_1 \wedge s_2)
\end{align*}
\end{theorem}

\begin{proof}
  This result was proved in \cite{alzati_bolognesi} for $C$ non hyperelliptic, but the proof extends harmlessly to the hyperelliptic case.
  The Picard variety $\Pic^{g-1}(C)$ contains a model $\widetilde{C}$ of $C$, made up by line bundles of type $\O(B + p)$, with $p \in C$. The span of $\widetilde{C}$ inside $|2 \Theta|^*$ corresponds to the complete linear system $|2 D|^*$. 
  Moreover, the linear span of $\widetilde{C}$ is the annihilator of $\P_c$. In particular, the projection $\pproj|_{\theta(\genSU)}$ determines a hyperplane in the annihilator of $\P_c$, which is a point in $|2 D|$. We have that this projection is given as
\begin{align*}
  \pproj|_{\theta(\genSU)} : \theta(\genSU) & \to |2D|  \\
  \theta(E) & \mapsto \Delta(E) ,
\end{align*}
where $\Delta(E)$ is the divisor defined by
\begin{align}
  \Delta (E):= \{p \in C \ | \ h^0(C,E \otimes \O(B + p)) \not = 0 \}.
\end{align}
  Equivalently, we have that $\Delta(E) = \theta(E) \cap \widetilde{C}$.
  Since $\theta(E) = \theta(i^*E)$, we directly obtain from this equation that $\Delta(E) = \Delta(i^* E)$.
  Finally, an easy Riemann-Roch argument shows that that $\Delta(E)$ is the divisor of zeroes of $s_1 \wedge s_2$.
\end{proof}

  Recall that the linear system $|K + 2D|$ embeds the curve $C$ in the projective space $\PD$. Let $N \in |2D|$ and consider the linear span $\langle N \rangle \subset \PD$. The annihilator of $\langle N \rangle$ is the vector space $H^0(C, 2D + K - N)$, which has dimension $g$. In particular, the linear span $\langle N \rangle$ has dimension $(3g - 2) - g = 2g - 2$.
  Let us write $$\PN := \langle N \rangle \subset \PD.$$ 
  We will study the classifying map $\p$ by means of the restriction maps $\p|_{\PN}$ when $N$ vary in the linear system $|2D|$.

\begin{notation}
  For simplicity, let us write $\pPN$ for the restriction map $\p|_{\PN}$.
\end{notation}
\begin{proposition} \label{prop:fibre=image}
  Let $N$ in $|2D|$ be a general divisor on $C \subset \PD$. Then, the image of
\begin{align*}
  \pPN : \PN \tto \theta(\SU)
\end{align*}
is the closure in $\theta(\SU)$ of the fiber over $N \in |2D|$ of the projection $\pproj$.
\end{proposition}

\begin{proof}
  Let $(e) \in \PD$ be an extension 
\begin{align*}
(e) \quad  0 \rightarrow \O(-D) \xrightarrow{i_e} E_e \xrightarrow{\pi_e} \O(D) \to 0.
\end{align*}
  By \cite[Proposition 1.1]{lange_narasimhan}, we have that $e \in \PN$ if and only if there exists a section 
$$\alpha \in H^0(C, \Hom(\O(-D),E))$$
such that $\Zeroes(\pi_e \circ \alpha) = N$. 
  This means that $\alpha$ and $i_e$ are two independent sections of $E_e \otimes \O(D)$ with $\Zeroes(\alpha \wedge i_e) = N$. Consequently, $\theta(E_e) = \pPN (e)$ is projected by $\pproj$ on $N \in |2D|$ by Theorem \ref{th:proyeccion}. Hence, the image of $\pPN$ is contained in $\overline{\pproj^{-1}(N)}$.

  Conversely, by the proof of Theorem \ref{th:proyeccion}, we have that for every bundle $E \in \genSU$, $\theta(E)$ is projected by $\pproj$ to a divisor $\Delta(E) \in |2D|$.
  The argument used above implies that the fiber $\p^{-1}(\theta(E)) = f_D^{-1}(E)$ of such bundle is contained in $\P^{2g - 2}_{\Delta(E)}$. Consequently, the fiber of a general divisor $N \in |2D|$ by $\pproj$ is contained in the image of $\pPN$.
\end{proof}
  
{}\section{The restriction map} \label{sec:restriction_map}
  In this section, we describe the restriction map $\pPN$  for a generic $N \in |2D|$. Let us define the following secant varieties in $\PD$:
\begin{align*}
  \Sec^N &:= \Sec^{g-2}(C) \cap \PN \\
  \Sec^n(N) &:= \bigcup_{\substack{M \subset N \\ \#M = n + 1}} \operatorname{span} \{M \}
\end{align*}
  Note that, since the points of $N$ are already in $\PN$, we have $\Sec^{g-2}(N) \subset \Sec^N$. 
\begin{lemma} \label{lem:base_locus}
  The restriction map $\pPN$ is given by the linear system $|\I_{\Sec^N}(g)|$. 
\end{lemma}
\begin{proof}
  This is a direct consequence of Theorem \ref{th:bertram} and Proposition \ref{prop:linear_systems_coincide}.
\end{proof} 

\subsection{The map $h_N$} \label{sec:h_N}
  Let $h_N$ be the rational map defined by the complete linear system $|\I_{\Sec^{g-2}(N)}(g)|$. 
  In \cite{bolognesi_BLMS}, the author describes this map and relates it to the GIT and Mumford-Knudsen compactifications of the moduli space $\M_{0,2g}$.
  In the next paragraphs we briefly recall some of these results.

  Let $Z$ be a rational normal curve in $\PN$ passing through the points of $N$. This curve uniquely determines a configuration of $2g$ points in $\P^1$. 
  
  The following Theorem is a translation of results from \cite[Theorem 4.3 and Proposition 4.5]{bolognesi_BLMS} to our situation.

\begin{theorem} \label{th:RNC}
  The image of $h_N$ is canonically isomorphic to the GIT moduli space $\MGIT$ of ordered configurations of $2g$ points in $\P^1$. The map $h_N$ contracts every rational normal curve $Z$ passing through the $2g$ points $N$ to a point $z$ in $\MGIT$. This point represents an ordered configuration of the $2g$ points $N$ on the rational curve $Z$.
\end{theorem}
  Since the map $h_N$ is defined by a linear subsystem of $|\I_{\Sec^N}(g)|$, we have that $\pPN$ factors through $h_N$:
\begin{equation} \label{diag:factorizacionporhn}
  \begin{tikzcd}
    \PN \arrow[r,dashed,"h_N"] \arrow[rd,swap,dashed,"\pPN"] & \MGIT  \arrow[d,dashed] \\
    &                                           {|\I_{\Sec^N}(g)|} % m = 2g - 3
  \end{tikzcd}
\end{equation}

  Let us now study the complete base locus $\Sec^N$ of the restriction map $\pPN$.
  By definition, the points in $\Sec^N$ are given by the intersections $\langle L_{g-1} \rangle \cap \PN $, where $L_{g-1}$ is an effective divisor of degree $g-1$ and $\langle L_{g-1} \rangle$ is its linear span in $\PD$.
  If $L_{g-1}$ is contained in $N$, it is clear that $\langle L_{g-1} \rangle \subset \Sec^{g-2}(N) \subset \PN$.

\begin{lemma} \label{lemma:nonempty_intersection}
  Let $L_{g-1}$ be an effective divisor on $C$ of degree $g - 1$, not contained in $N$. Then, 
\begin{align*}
  \langle L_{g-1} \rangle \cap \PN \not = \phi \qquad  \text{if and only if} \qquad \dim|L_{g-1}| \geq 1.
\end{align*}
  Moreover, if the intersection is non-empty, we have that $$\dim (\langle L_{g-1} \rangle \cap \PN) = \dim |L_{g-1}| - 1.$$
\end{lemma}

\begin{proof}
  First, let us suppose that $L_{g-1}$ and $N$ have no points in common. 
  The vector space $V := H^0(C, 2D + K - L_{g-1})$ is the annihilator of the span $\langle L_{g-1} \rangle$ in $\PD$. By the Riemann-Roch theorem, we see that $V$ has dimension $2g$, hence $$\dim \langle L_{g-1} \rangle = (3g - 2) - 2g = g - 2.$$
  Let $d$ be the dimension of the span $\langle L_{g-1}, N \rangle$ of the points of $L_{g-1}$ and $N$. Since the dimension of $\PN = \langle N \rangle$ is $2g-2$, we have that $d \leq (g-2) + (2g - 2) + 1 = 3g-3$, where the equality holds iff $\langle L_{g-1} \rangle \cap \PN$ is empty.

  In particular, this intersection is non-empty iff $d \leq 3g -4$. Since $\dim |K + 2D|^* = \dim \PD = 3g - 2$, this is equivalent to the annihilator space $$W :=  H^0(C, 2D + K - L_{g-1} - N) = H^0(C, K - L_{g-1})$$ being of dimension $\geq 2$. By Riemann-Roch and Serre duality, we obtain that this condition is equivalent to $\dim |L_{g-1}| \geq 1$. 

  More precisely, let us suppose that $\langle L_{g-1} \rangle \cap \PN$ is non-empty and let $e := \dim (\langle L_{g-1} \rangle \cap \PN)$. Then, we have that $$d = 3g - 3 - (e + 1),$$ and the annihilator space $W$ is of dimension $2 + e$. Again by a Riemann-Roch computation, we conclude that $e = \dim |L_{g-1}| - 1$.

  Finally, if $L_{g-1}$ and $N$ have some points in common, we have to count them only once in the vector space $W$ to avoid requiring higher vanishing multiplicity to the sections.
 
\end{proof}

  From this Lemma, we conclude that $\Sec^{g-2}(N)$ is a proper subset of $\Sec^N$ if and only if there exists a divisor $L_{g-1}$ not contained in $N$ with $\dim|L_{g-1}| \geq 1$. By the Existence Theorem of Brill-Noether theory (see \cite[Theorem 1.1, page 206]{arbarello_cornalba}) this is equivalent to $g \geq 4$ in the non-hyperelliptic case. We will discuss the first low genera cases in Section \ref{sec:lowgenera}.

\subsection{A rational normal curve} \label{sec:gamma}

  We have seen that the secant variety $\Sec^{g-2}(N)$ is part of the base locus $\Sec^N$ of the map $\pPN$, and that this inclusion is strict for $g \geq 4$ in the non-hyperelliptic case.

  In the hyperelliptic case, we have an additional base locus for every genus, which appears due to the hyperelliptic nature of the curve. This locus arises as follows: for each pair $P = \{ p, i(p) \}$ of involution-conjugate points in $C$, consider the hyperelliptic secant line $l$ in $\PD$ passing through the points $p$ and $i(p)$. Let $Q_P$ be the intersection of the line $l$ with $\PN$.
  Let us define $\Gamma \subset \PN$ as the locus of intersection points $Q_P$ when we vary the pair $P$.

\begin{figure}
      \centering
          \def\svgwidth{0.7\columnwidth}
          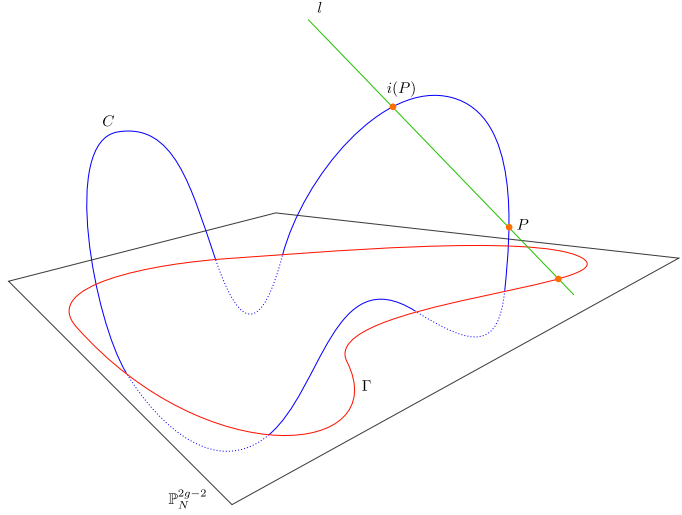
	\caption{The situation in genus $3$. The curves $\Gamma$ and $C$ intersect along the divisor $D$, of degree 6. The secant lines $l$ cutting out the hyperelliptic pencil define the curve $\Gamma$.}\label{fig:gamma}
\end{figure}

\begin{lemma} \label{lem:gamma}
  The locus $\Gamma \subset \PN$ is a rational normal curve in $\PN$. Moreover, $\Gamma$ passes through the $2g$ points $N \subset C$. 
\end{lemma}
\begin{proof}
  Let us start by showing that the intersection is non-empty for every pair $\{p, i(p) \}$. Since $\dim |p + i(p)| = \dim |h| = 1$, the intersection $l \cap \PN$ is non-empty by Lemma \ref{lemma:nonempty_intersection}. 

  Let us show that this intersection is a point, i.e. that the line $l$ is not contained in $\PN$.
  Recall that $\PD = |2D + K|^*$.
  If the points $p$ and $i(p)$ are both not contained in the divisor $N$, the vector space $$ V := H^0(C,2D + K - N - (p + i(p)) = H^0(C,2D + K - N - h)$$ is exactly the annihilator of the span $\langle l, \PN \rangle$ in $\PD$. In particular, the codimension of $\langle l, \PN \rangle$ in $\PD$ is the dimension of $V$. By Riemann-Roch and Serre duality, we get that $\dim V = g - 2$, thus $\dim \langle l, \PN \rangle = 3g - 2 - (g - 2) = 2g$. This means that the intersection  $l \cap \PN$ is a point. 

  For the case $p \in N$ and $i(p) \not \in N$, let us remark that the the annihilator of the span $\langle l, \PN \rangle$ is now the vector space $H^0(C,2D + K - N - i(p))$. Since $$h^0(C,2D + K - N - i(p)) < h^0(C,2D + K - N),$$ we conclude that the line $l$ is not contained in $\PN$.
  The case $\{p, i(p) \} \subset N$ is excluded by our genericity hypotheses on $N$.  
  To summarize, we deduce that the locus $\Gamma$ is a curve in $\PN$.

  Let $q$ be a point of $N$. Then, $q$ is a point of the plane $\PN$. Consequently, the line passing through $q$ and $i(q)$ intersects the plane $\PN$ at $q$. Thus, we have that $\Gamma$ passes through the points of $N$.
  Moreover, it is clear that $N$ is the only intersection of $\Gamma$ and $C$, i.e. $\Gamma \cap C = N$.

  Let us prove now that $\Gamma$ is a rational normal curve. Since $\Gamma$ is defined by the hyperelliptic pencil, it is clear that $\Gamma$ is rational. Moreover, since the divisor $D$ is generic, the span of any subset of $2g - 1$ points of $D$ is $\PN$. Thus, it suffices to show that the degree of $\Gamma \subset \PN$ is precisely $2g - 2$.
  Let us write 
$$N = q_1 + \cdots + q_{2g}$$
 with $q_1, \ldots, q_{2g} \in C$.
  By the previous paragraph, $\Gamma$ passes by these $2g$ points. Let us consider a hyperplane $H$ of $\PN$ spanned by $2g - 2$ points of $N$. Without loss of generality, we can suppose that these points are the first $2g - 2$ points $q_1, \ldots, q_{2g - 2}$. To show that the degree of $\Gamma$ is $2g - 2$, we have to show that the intersection of $\Gamma$ with $H$ consists exactly of the points $q_1, \ldots, q_{2g - 2}$.

  The intersection $l \cap H$ is empty if and only if the linear span $\langle l, H \rangle$ of $l$ and $H$ in $\PD$ is of maximal dimension $2g - 1$, i.e. of codimension $g - 1$ in $\PD$. 
  Consider the divisors 
\begin{align*}
  D_H= q_1 + \cdots + q_{2g - 2} \qquad \mbox{and} \qquad D_l = q + i(q) \ .
\end{align*}
  As before, if $\{p, i(p)\} \cap \{q_1, \ldots, q_{2g - 2} \}$ is empty, the vector space $W = H^0(C,2D + K - D_H - D_l)$ is the annihilator of the span $\langle l, H \rangle$ in $\PD$.
  In particular, the codimension of $\langle l, H \rangle$ in $\PD$ is given by the dimension of $W$. Again by Riemann-Roch and Serre duality theorems, we can check that  
  $$\dim W = h^0(C,-2D + D_H + D_l) + g - 1.$$
  Thus, the codimension of $\langle l, H \rangle$ in $\PD$ is greater than $g - 1$ if and only if $h^0(C,-2D + D_H + D_l) > 0$. Since $\deg (-2D + D_H + D_l) = 0$, this is equivalent to $-2D + D_H + D_l \sim 0$. Since $N = q_1 + \cdots + q_{2g} \sim 2D$, we have that 
\begin{align*}
  -2D + D_H + D_l \sim 0 &\iff p + i(p) \sim q_{2g - 1} + q_{2g} \\ &\iff h \sim q_{2g - 1} + q_{2g} \\ &\iff i(q_{2g - 1}) = q_{2g}.
\end{align*}
  By our genericity hypothesis on $N$, the last condition is not satisfied.
  Consequently, we conclude that the line $l$ intersects the hyperplane $H$ iff $\{p, i(p)\} \cap \{q_1, \ldots, q_{2g - 2} \}$ is non-empty, i.e. iff $p$ or $i(p)$ is one of the $q_k$ for $k = 1, \ldots, 2g - 2$. In particular, $$\Gamma \cap H = \{q_1, \ldots, q_{2g - 2} \}$$
as we wanted to show.

\end{proof}
  Hence, the curve $\Gamma$ is contracted by the map $h_N$ to a point $P_N \in \MGIT$. The point $P_N$ represents a hyperelliptic curve $C_N$ of genus $g-1$ together with an ordering of the Weierstrass points $N$ on the rational curve $\Gamma$.

  Let us now illustrate the geometric situation by explaining in detail the first case in low genus.

{}\section{The case $g = 3$} \label{sec:g3}
  Let $C$ be a hyperelliptic curve of genus 3. 
  In this setting, we have that the map $\theta$ factors through the involution $i^*$, and embeds the quotient $\SU / {\langle i^* \rangle }$ in $\P^7 = |2 \Theta|$ as a quadric hypersurface (see \cite{beauville_rang2} and \cite{desale_ramanan}).
  Let $D$ be a general effective divisor of degree 3. 
  The projective space $\P^7_D$, as defined in Section \ref{sec:introduction},  parametrizes the extension classes in $\Ext^1(\O(D), \O(-D))$.
  The classifying map $\p$ is given in this case by the complete linear system $|\I_C^2(3)|$ of cubics vanishing on $C$ with multiplicity 2. According to Proposition \ref{prop:linear_systems_coincide}, this linear system coincides with the linear system $|\I_{\Sec^1}(3)|$ of cubics vanishing along the secant lines of $C$. Recall from Section \ref{sec:gamma} that among these lines we have the  secant lines $l$ passing through involution-conjugate points. These form a pencil given by the linear system $|h|$.

  The image of the projection of $\theta(\genSU)$ from $\P_c = \P^3 \subset |2 \Theta|$ is also a $\P^3$, that is identified with $|2 D|$ by Theorem \ref{th:proyeccion}. Let $N \in |2D|$ be a generic reduced divisor. By Proposition \ref{prop:fibre=image}, the closure of the fiber $\pproj^{-1}(N)$ is the image via $\p$ of the $\PcuatroN$ spanned by the six points of $N$. 

\subsection{The restriction to $\P_N^4$} \label{sec:factor_g3}
  The base locus of the restriction map $\pPN = \pPcuatroN$ is $\Sec^N = \Sec^1(C) \cap \P^4_N$ by Lemma \ref{lem:base_locus}. 
  The secant variety $\Sec^1(N) \subset \Sec^N$ is the union of the 15 lines passing through pairs of the 6 points of $N$. According to Lemma \ref{lemma:nonempty_intersection}, the further base locus $\Sec^N \setminus \Sec^1(N)$ is given by the intersections of $\P^4_N$ with the lines spanned by degree 2 divisors $L_2$ on $C$ not contained in $N$ satisfying $\dim|L_2| \geq 1$. By Brill-Noether theory, there is only one linear system of such divisors on a genus 3 curve, namely the hyperelliptic linear system $|h|$ (see, for example, \cite{arbarello_cornalba}, Chapter V). We will review these ideas in Section \ref{sec:lowgenera}). This linear system defines, by the intersections with $\PN$ of the lines spanned by the hyperelliptic pencil, the curve $\Gamma$ that we introduced in Section \ref{sec:gamma}.
  Hence, we have that $\Sec^N = \{15 \text{ lines} \} \cup \Gamma$, and the restriction map $\pPN$ factors as
\begin{equation*}
  \begin{tikzcd}
    \P^4_N \arrow[r,dashed,"h_N"] \arrow[rd,swap,dashed,"\pPN"] & \MGITsix \subset \P^4 \arrow[d,dashed,"p"] \\
    &                                           \P^{3}
  \end{tikzcd}
\end{equation*}
where $h_N$ is the map defined by the complete linear system $|\I_{\Sec^1(N)}(3)|$ of cubics vanishing along the 15 lines defined by the points of $N$, and $p$ is the projection with center the image via $h_N$ of the further base locus $\Gamma$.
  According to Theorem \ref{th:RNC}, the image of $h_N$ is the GIT moduli space $\MGITsix$ if $N$ is generic and reduced. It is a classic result that this moduli space is embedded in $\P^4$ as the Segre cubic $S_3$ (see for instance \cite{dolgachev_ortland}). This 3-fold arises by considering the linear system of quadrics in $\P^3$ that pass through five points in general position, thus it is isomorphic to the blow-up of $\P^3$ at these points, followed by the blow-down of all lines joining any two points.
  The curve $\Gamma \subset \P^4_N$ is a rational normal curve by Lemma \ref{lem:gamma}, hence $\Gamma$ is contracted to a point $P$ by $h_N$ again by Theorem \ref{th:RNC}.

  By \cite{bertram} and Lemma \ref{lem:base_locus}, the linear system $|\OO_{S_3}(1)|$ of hyperplanes in $S_3$ is pulled back by $h_N$ to $|\I_{\Sec^1(N)}(3)|$ on $\PcuatroN$. The linear system $|\OO_{S_3}(1) - P|$ of hyperplanes in $S_3$ passing through $P$ is pulled back to the complete linear system $|\I_{\Sec^N}(C)|$ defining $\pPN$. 
  Hence, the map $p$ is the linear projection with center $P$.
  Since $S_3$ is a cubic, the projection $p$ is a 2:1 map. We will see in the next Section that this will be also the case for higher genus.

  The point $P$ in $\MGITsix$ represents a rational curve with 6 marked points.
  Let $C'$ be the hyperelliptic genus 2 curve constructed as the 2:1 cover of this rational curve ramified in these 6 points.
  According to Theorem 4.2 of \cite{kumar}, the Kummer variety $\Kum(C')$ is contained in the image of $p$, and it is precisely the branching locus of $\pi$.

  Let us recall that, by definition, the map $\p$ factors globally through the 2:1 map $\theta$ introduced in Section \ref{sec:moduli}, and the preimages of the map $\theta$ are of the form $E$, $i^*E$. We will see in Section \ref{sec:general_case} that the map $p$ defined above is exactly the restriction of the map $\theta$ to the image of $\P_N^4$ by $f_D$. 
  In particular, the ramification locus of the map $\theta$ is birational to a fibration over $|2D| \cong \P^3$ in Kummer varieties of dimension 2.
  Moreover, these facts also hold for higher genus, except that we will have to compose with a birational morphism (see Theorem \ref{th:fibration_kummer}). 

\subsection{The global map $\p$}

  In the genus 3 setting, the linear system $|2D|$ is a $\P^3$. 
  By Proposition \ref{prop:fibre=image}, the image of $\PcuatroN$ by $\p$ is the closure of the fiber $\pproj^{-1}(N)$. 
  For each point $N$ in $|2D|$, this image is $\P^3 = |\I^2_C(3)|^*$, which is the image of the Segre variety $\MGITsix$ under the projection with center $P$.
   Thus, the image of the global map $\p$ birational to a $\P^3$-blundle over $|2D| = \P^3$. In fact, this image is also a quadric hypersurface in $\P^7$ \cite{desale_ramanan}.

{}\section{The general case} \label{sec:general_case}
  Now let us describe the geometric situation for a hyperelliptic curve $C$ of arbitrary genus.
  We start by showing that the factorization of the restriction maps $\pPN$ outlined in the lasts Sections holds also in the global setting.
  
\subsection{The global factorization} \label{sec:Stein_factorization}
  The map $\p$ is not everywhere defined, since some of the bundles sitting in the sequences classified by $\PD$ are unstable.
  In \cite{bertram}, the author constructs the resolution $\widetilde{\p}$ of the map $\p$ as a sequence of blow-ups
\begin{equation*} 
  \begin{tikzcd}
    \widetilde{\PD} \arrow[d, "\bl_{g-1}", swap]   \arrow[rdd, "\widetilde{\p}"] \\
    \vdots \arrow[d, "\bl_1", swap] \\
    \PD  \arrow[r, dashed, "\p"] & {|}2 \Theta {|}
  \end{tikzcd}
\end{equation*} 
along certain secant varieties $C = \Sec^0(C) \subset \Sec^1(C) \subset \cdots \subset \Sec^{g-1}(C) \subset \PD$.
  This chain of morphisms is defined inductively as follows:
  the center of the first blow-up $\bl_1$ is the curve $C = \Sec^0(C)$. 
  For $k = 2, \ldots, g-1$,
  the center of the blow-up $\bl_k$ is the strict transform of the secant variety $\Sec^{k-1}(C)$ under the blow-up $\bl_{k-1}$.
  
  Recall that the map $\p$ is the composition of the classifying map $f_D$ and the 2:1 map $\theta$. Thus, the map $f_D$ lifts to a map $\widetilde{f_D}$ which makes the following diagram commute:
\begin{equation} \label{diag:global}
  \begin{tikzcd}
    \widetilde{\PD} \arrow[r,"\widetilde{f_D}"] \arrow[rd,swap,"\widetilde{\p}"] & \SU \arrow[d,"\theta"] \\
    &                               {|}2 \Theta {|}            
  \end{tikzcd}
\end{equation}

  When $C$ has genus 3 we have already described a factorization of the restriction map $\widetilde{\p}|_{\PN}$, that was constructed in Section \ref{sec:factor_g3}.
  We will show in the next section that the factorization of Section \ref{sec:factor_g3} is actually the fiberwise version of Diagram \ref{diag:global}, i.e. when we fix $N \in |2 D|$ generic and we consider the restriction $\pPN$. Recall that the image of $\widetilde{f_D}|_{\PN}$ inside $\SU$ is the fiber of $\pproj$ over $N$ (see Theorem \ref{th:proyeccion} and Proposition \ref{prop:fibre=image}).
  This remains true for $g \geq 4$, as we will see after extending the description of Section \ref{sec:factor_g3} to higher genus.

\subsection{Osculating projections}
  We recall here a generalization of linear projections that will allow us to describe the map $p$ in higher genus. For a more complete reference, see for example \cite{massarenti_rischter}.
  Let $X \subset \P^N$ be an integral projective variety of dimension $n$, and $p \in X$ a smooth point. Let
\begin{align*}
  \phi: \mathcal{U} \subset \mathbb{C}^n &\longrightarrow \mathbb{C}^N \\
  (t_1, \ldots, t_n) &\longmapsto \phi(t_1, \ldots, t_n)
\end{align*}
be a local parametrization of $X$ in a neighborhood of $p = \phi(0) \in X$. For $m \geq 0$, let $O^m_p$ be the affine subspace of $\mathbb{C}^N$ passing through $p \in X$ and generated by the vectors $\phi_I(0)$, where $\phi_I$ is a partial derivative of $\phi$ of order $\leq m$.  

  By definition, the \emph{$m$-osculating space} $T_p^m X$ of $X$ at $p$ is the projective closure in $\P^N$ of $O^m_p$. 
  The \emph{$m$-osculating projection} $$\Pi_p^m:X \subset \P^N \dashrightarrow \P^{N_m}$$ is the corresponding linear projection with center $T_p^m$.

\subsection{The further base locus of $\pPN$ and the map $\pi_N$} \label{sec:furtherbaselocus} 
  We define the \emph{further base locus} of $\pPN$ as the set
\begin{align*}
  {\Sec^N}' := \Sec^N \setminus \{ \Gamma \cup \Sec^{g-2}(N) \}.
\end{align*}
  This locus is non-empty for $g \geq 4$ due to the existence of effective divisors $L_{g-1}$ in the conditions of Lemma \ref{lemma:nonempty_intersection}, as we will see in Section \ref{sec:lowgenera}.

\begin{lemma} \label{lemma:vanishing_forms}
  Let $Q$ be a $r$-form in $\P^n$ vanishing at the points $P_1$ and $P_2$ with multiplicity $l_1$ and $l_2$ respectively. Then, $Q$ vanishes on the line passing through $P_1$ and $P_2$ with multiplicity at least $l_1 + l_2 - r$.
\end{lemma}
\begin{proof}
  See, for example, \cite[page 2]{kumar_linearsystems}.
\end{proof}
 
  Let $\N_1 := |\I_{\Sec^N}(g)|$ be the linear system that defines $\pPN$. The forms in $\N_1$ vanish with multiplicity $g - 1$ along the points of $C$ (see Lemma \ref{prop:linear_systems_coincide}). By Lemma \ref{lemma:vanishing_forms}, these forms vanish then with multiplicity $(g-1) + (g-1) - g = g-2$ along the secant lines $l$ cutting out the hyperelliptic pencil. Thus, these forms vanish with multiplicity $g-2$ on the curve $\Gamma$. 
  Let us also define the linear system $\N_3 := |\I_{\Sec^{g-2}(N)}(g)|$, and consider the partial linear system $\N_2 \subset \N_3$ of forms vanishing (with multiplicity 1) along $\Sec^{g-2}(N)$, and with multiplicity $g - 2$ on $\Gamma$. By our previous observation, we have the following inclusions of linear systems:
\begin{align*}
  \N_1 \subset \N_2 \subset \N_3.
\end{align*}
  These inclusions yield a factorization 
\begin{equation} \label{diag:Kumar}
  \begin{tikzcd}
    \PN \arrow[r,dashed,"h_N"] \arrow[rrd,swap,dashed,"\pPN"] & \MGIT \subset {\N_3}^*  \arrow[r,dashed,"\pi_N"] &  {\N_2}^* \arrow[d,dashed,"l_N"] \\
    &                       &                    {\N_1}^*
  \end{tikzcd}
\end{equation} 
  The first map $h_N$ is the one defined in Section \ref{sec:h_N}, its image is the moduli space $\MGIT$.
  According to Theorem \ref{th:RNC}, this map contracts the curve $\Gamma$ to a point $h_N(\Gamma)$. 
  The map $\pi_N$ coincides with the map $p$ defined in Section \ref{sec:factor_g3} when $g = 3$.

\begin{proposition} \label{prop:fact_restringida}
  The map $\pi_N$ is the $(g-3)$-osculating projection $\Pi^{g-3}_{P}$ with center the point $P = h_N(\Gamma)$.
\end{proposition}
\begin{proof}
  From the definition of the linear systems $\N_2$ and $\N_3$, the base locus of the map $\pi_N$ is the point $h_N(\Gamma)$ . In particular, the map $\pi_N$ is an osculating projection of some order with respect to this point. Since the forms in $\N_2$ vanish with multiplicity $g - 2$ along $\Gamma$, the order the projection $\pi_N$ is $g-3$.  
\end{proof}

\subsection{The Kumar factorization} \label{sec:kumarfactorization}
  By the results of Section \ref{sec:h_N}, the map $h_N$ contracts the curve $\Gamma$ to a point $P_N$ in $\MGIT$ representing an ordered configuration of the $2g$ marked points $N$. This point in turn corresponds to a hyperelliptic genus $(g - 1)$ curve $C_N$ together with an ordering of the Weierstrass points.
  In the paper \cite{kumar}, Kumar describes a generalization of the construction outlined in Section \ref{sec:furtherbaselocus}. Let $\Omega$ be the linear system of $(g-1)$-forms in $\P^{2g-3}$ vanishing with multiplicity $g-2$ at $2g - 1$ general points $e_1, \ldots, e_{2g - 1} \in \P^{2g-3}$. Kumar shows that the rational map $i_{\Omega}$ induced by $\Omega$ maps birationally $\P^{2g - 3}$ onto the GIT quotient $\MGIT$. Furthermore, let $w = i_\Omega(e_0)$ be a general point in $\MGIT$. This point represents a hyperelliptic curve $C_w$ of genus $g-1$, together with an ordering of the Weierstrass points. The partial linear system $\Lambda \subset \Omega$ of $(g-1)$-forms vanishing with multiplicity $g-2$ at all points $e_0, \ldots, e_{2g - 1} \in \P^{2g-3}$ induces a rational projection $\kappa:\MGIT \to |\Lambda|^*$, and Kumar shows that this map is 2:1 onto a connected component of the moduli space $\SUw$ of rank 2 semistable vector bundles with trivial determinant over the curve $C_w$. Furthermore, Kumar proves that the map $\kappa$ is ramified along the Kummer variety $\Kum(C_w) \subset \SUw$:
\begin{equation*}
  \begin{tikzcd}
    \P^{2g-3} \arrow[r,dashed,"i_\Omega"] \arrow[rd,swap,dashed,"i_\Lambda"] & \MGIT\arrow[d,dashed,"\kappa"] \\
    &                                          \SUw 
  \end{tikzcd}
\end{equation*}

\begin{theorem}  \label{th:pi=kappa}
  The map $\pi_N$ coincides with the Kumar map $\kappa$. In particular, the map $\pi_N$ is 2:1.
\end{theorem}
\begin{proof}
  Let $\mathcal{S} = \MGIT$. As in Section \ref{sec:factor_g3}, the linear system $|{\mathcal{O}}_{\mathcal{S}}(1)|$ of hyperplanes in $\mathcal{S}$ is pulled back by $h_N$ to the linear system $\N_3 = |\I_{\Sec^{g-2}(N)}(g)|$. 
  According to Proposition \ref{prop:fact_restringida}, the linear system $|\OO_{\mathcal{S}}(1) - (g - 2)P|$ of hyperplanes in $\mathcal{S}$ vanishing in $P$ with multiplicity $g-2$ is pulled back to the linear system $\N_2$ defining $\pi_N \circ h_N$. 
  Recall that the map $\kappa$ is also given by the linear system of hyperplanes in $\mathcal{S}$ vanishing in $P$ with multiplicity $g-2$. In particular, the map $\kappa$ coincides exactly with the map $\pi_N$ up to birationality.
\end{proof}
  
  We will show in the next Section that the map $l_N$ is actually birational, and that the map $\pi_N$ coincides with the restriction of the map $\theta$.

\subsection{The global description} \label{sec:the_global_description}

  The resolution map $\widetilde{\p}$ of $\p$ factors through the map $\theta$ as shown in Diagram \ref{diag:global}. When restricted to $\PN$ for a divisor $N$, we have also shown that the restriction map $\pPN$ factors through the 2:1 map $\pi_N$. Now we link these two factorizations:
\begin{theorem}
  Let $N \in |2D|$ be a generic effective divisor. Then, the restriction map $\theta|_{f_D\left(\PN\right)}$ is the map $\pi_N$ modulo composition with a birational map. 
\end{theorem}
\begin{proof}
  Let us place ourselves on the open set $\genSU \subset \SU$ of general stable bundles. One observes that the factorization $\widetilde{\p} =\theta \circ \widetilde{f_D}$ of Diagram \ref{diag:global} is the Stein factorization of the map $\widetilde{\p}$ along $\widetilde{\PD}$. 
  Indeed, the map $\theta$ is 2:1 as explained in Section \ref{sec:introduction}. Moreover, the preimage of a generic stable bundle $E$ by the map $f_D$ is the $\P^1$ arising as the projectivisation of the space of extensions of the form
\begin{align*}
  e: \quad  0 \to \O(-D)) \to E \to \O(D) \to 0.
\end{align*}
  In particular, the fibers of $\widetilde{f_D}$ over $\genSU$ are connected.  

  The restriction of $\widetilde{\p}$ to $\PN$ factors through the maps $h_N$ and $\pi_N$ (see Diagram \ref{diag:Kumar}), followed by the map $l_N$. 
  According to Theorem \ref{th:RNC}, the fibers of $h_N$ are rational normal curves, thus connected. Moreover, the map $\pi_N$ is 2:1 by Theorem \ref{th:pi=kappa}. 
  By unicity of the Stein factorization, we have our result. 

  Comparing with the factorization  $\widetilde{\p} =\theta \circ \widetilde{f_D}$, we see that $l_N$ cannot have relative dimension $> 0$. Hence, $l_N$ is a finite map. Since the degree of the map $\theta$ in the Stein factorization is 2, which is equal to the degree of $\pi_N$, we have that $l_N$ cannot have degree $> 1$.
  In particular, we have that the map $l_N$ is a birational map. 
\end{proof}

  From this description and the arguments of Section \ref{sec:kumarfactorization} yields the following result:
\begin{theorem} \label{th:fibration_kummer}
  The ramification locus of the map $\theta$ is birational to a fibration over $|2D| \cong \P^g$ in Kummer varieties of dimension $g - 1$.
\end{theorem}

{}\section{Further base locus} \label{sec:lowgenera}
  Recall from Section \ref{sec:restriction_map} that the base locus of the restriction map $\pPN$ is the intersection $ \Sec^N = \Sec^{g-2}(C) \cap \PN$ of the secant variety of $(g-2)$-dimensional secant planes of $C$ with $\PN$. 
  The subvarieties $\Sec^{g-2}(N)$ and $\Gamma$ of $\Sec^N$ yield the factorization of $\pPN$ through the maps $h_N$ and $\pi_N$ of Proposition \ref{prop:fact_restringida}.
  Let us now describe the set of further base locus
\begin{align*}
  {\Sec^N}' = \Sec^N \setminus \{ \Gamma \cup \Sec^{g-2}(N) \}.
\end{align*}
  This set is empty for $g = 3$, hence the map $\pPN$ is exactly the composition of $h_N$ and $\pi_N$, as described in Section \ref{sec:g3}.
  In higher genus, the existence of non-empty additional base locus $\Sec^N$ corresponds to the fact that, in higher genus, the map $\pPN$ is not exactly the composition of the maps $h_N$ and $\pi_N$. In other words, the map $l_N$ is non-trivial in higher genus.

  This supplementary base locus is given by the intersections of $(g-2)$-dimensional $(g-1)$-secant planes of $C$ in $\PD$ with $\PN$ out of $\Sec^{g-2}(N)$ and $\Gamma$. According to Lemma \ref{lemma:nonempty_intersection}, these intersections are given by effective divisors $L_{g-1}$ on $C$ of degree $g-1$, not contained in $\PN$, and satisfying $\dim |L_{g-1}| \geq 1$. Also by Lemma \ref{lemma:nonempty_intersection}, we obtain $\dim (\langle L_{g-1} \rangle \cap \PN)=\dim |L_{g-1}| - 1$.

  We will now give account of the situation in low genera. 

\subsubsection*{Case $g = 4$} 
  In this case, the divisor $N$ is of degree 8 and the map
\begin{align*}
  \p|_N : \P_N^6 \subset \P_D^{10} \tto |2 \Theta| = \P^{15}
\end{align*} 
is given by the linear system $|\I_C^3(4)|$. This map factors through the map $\pi_N$ which coincides with the $1$-osculating projection $\Pi^1_P$, where $P = h_N(\Gamma)$.
  
  We are looking for degree 3 divisors $L_3$ with $\dim |L_3| \geq 1$. These satisfy all $\dim |L_3| = 1$ and are of the form 
\begin{align*}
  L_3 = h + q \qquad \text{for } q \in C,
\end{align*}
where $h$ is the hyperelliptic divisor. Let $p$ be a point of $C$. Then $L_3 = p + i(p) + q$. Since $\dim |L_3| = 1$, the secant plane $\P^2_{L_3}$ in $\P_D^{10}$ spanned by $p$, $i(p)$ and $q$ intersects $\P^6_N$ in a point. But this point necessarily lies in $\Gamma$, since the line passing through $p$ and $i(p)$ is already contained in this plane. Hence, we do not obtain any additional locus, and the map $l_N$ is the identity map. 

\subsubsection*{Case $g = 5$}
  In this case, the divisors $L_4$ of degree 4 are all of the form 
\begin{align*}
  L_4 = h + q + r \qquad \text{for } q, r \in C,
\end{align*}
and satisfy $\dim |L_4| = 1$. Thus, the corresponding secant $\P^3_{L_4}$ spanned by $p$, $i(p)$, $q$ and $r$ intersects $\P^{8}_N$ in a point. As before, this point lies in $\Gamma$, thus we do not obtain any additional locus.

\subsubsection*{Case $g = 6$}
  Here we have, as in the genus 5 case, the divisors of the form 
\begin{align*}
  L_3 = h + q \qquad \text{for } q \in C,
\end{align*}
which do not give rise to any additional base locus. But there is a new family of divisors 
\begin{align*}
  L_5 = 2h + r \qquad \text{for } r \in C.
\end{align*}
  These divisors satisfy $\dim |L_5| = 2$. In particular, the intersection of the $\P^4_{L_5}$ spanned by $p$, $i(p)$, $q$, $i(q)$ and $r$, for $p, q \in C$ with $\P^{10}_N$ is a line $m$ in $\P^{10}_N$. The line $l_1$ (resp. $l_2$) spanned by $p$ and $i(p)$ (resp. $q$, $i(q)$) intersects $\Gamma$ in a point $\widetilde{p}$ (resp. $\widetilde{q}$). In particular, the line $m$ is secant to $\Gamma$ and passes through  $\widetilde{p}$ and $\widetilde{q}$. Since every point of $\Gamma$ comes as an intersection of a secant line in $C$ with $\P^{10}_N$, we obtain the following description of the base locus of $\pPN$:
\begin{proposition}
  Let $C$ be a curve of genus $g = 6$. Then, the base locus of the restriction map $\pPN$ is the ruled 3-fold $\Sec^1(\Gamma)$.
\end{proposition}

\bibliography{mibib}

\begin{thebibliography}{10}

\bibitem{alzati_bolognesi}
A.~Alzati and M.~Bolognesi.
\newblock A structure theorem for {${\mathcal{SU}}_C(2)$} and the moduli of
  pointed rational curves.
\newblock {\em J. Algebraic Geom.}, 24(2):283--310, 2015.

\bibitem{arbarello_cornalba}
E.~Arbarello, M.~Cornalba, P.~A. Griffiths, and J.~Harris.
\newblock {\em Geometry of algebraic curves}.
\newblock Number v. 1 in Grundlehren der mathematischen Wissenschaften.
  Springer-Verlag, 1985.

\bibitem{atiyahvectorbundles}
M.~F. Atiyah.
\newblock Vector bundles over an elliptic curve.
\newblock {\em Proc. London Math. Soc. (3)}, 7:414--452, 1957.

\bibitem{beauville_rang2}
A.~Beauville.
\newblock Fibr\'es de rang {$2$} sur une courbe, fibr\'e d\'eterminant et
  fonctions th\^eta.
\newblock {\em Bull. Soc. Math. France}, 116(4):431--448 (1989), 1988.

\bibitem{beauville_vectorbundles}
A.~Beauville.
\newblock Vector bundles on curves and theta functions.
\newblock In {\em Moduli spaces and arithmetic geometry}, volume~45 of {\em
  Adv. Stud. Pure Math.}, pages 145--156. Math. Soc. Japan, Tokyo, 2006.

\bibitem{beauville_narasimhan_ramanan}
A.~Beauville, M.~S. Narasimhan, and S.~Ramanan.
\newblock Spectral curves and the generalised theta divisor.
\newblock {\em J. Reine Angew. Math.}, 398:169--179, 1989.

\bibitem{bertram}
A.~Bertram.
\newblock Moduli of rank-{$2$} vector bundles, theta divisors, and the geometry
  of curves in projective space.
\newblock {\em J. Differential Geom.}, 35(2):429--469, 1992.

\bibitem{bolognesi_conicbundle}
M.~Bolognesi.
\newblock A conic bundle degenerating on the {K}ummer surface.
\newblock {\em Math. Z.}, 261(1):149--168, 2009.

\bibitem{bolognesi_BLMS}
M.~Bolognesi.
\newblock Forgetful linear systems on the projective space and rational normal
  curves over {$\mathcal{M}^{\operatorname{GIT}}_{0,2n}$}.
\newblock {\em Bull. Lond. Math. Soc.}, 43(3):583--596, 2011.

\bibitem{bolognesi_brivio}
M.~Bolognesi and S.~Brivio.
\newblock Coherent systems and modular subvarieties of {$\mathcal{SU}_C(r)$}.
\newblock {\em Internat. J. Math.}, 23(4):1250037, 23, 2012.

\bibitem{brivio_verra}
S.~Brivio and A.~Verra.
\newblock The theta divisor of {${\mathcal{SU}}_C(2,2d)^s$} is very ample if
  {$C$} is not hyperelliptic.
\newblock {\em Duke Math. J.}, 82(3):503--552, 1996.

\bibitem{brivio_verra_plucker}
S.~Brivio and A.~Verra.
\newblock Pl\"ucker forms and the theta map.
\newblock {\em Amer. J. Math.}, 134(5):1247--1273, 2012.

\bibitem{casalaina-martin_gwena_teixidor}
S.~Casalaina-Martin, T.~Gwena, and M.~Teixidor~i Bigas.
\newblock Some examples of vector bundles in the base locus of the generalized
  theta divisor.
\newblock {\em C. R. Math. Acad. Sci. Paris}, 347(3-4):173--176, 2009.

\bibitem{desale_ramanan}
U.~V. Desale and S.~Ramanan.
\newblock Classification of vector bundles of rank {$2$} on hyperelliptic
  curves.
\newblock {\em Invent. Math.}, 38(2):161--185, 1976/77.

\bibitem{dolgachev_ortland}
I.~Dolgachev and D.~Ortland.
\newblock Point sets in projective spaces and theta functions.
\newblock {\em Ast\'erisque}, 165:210 pp. (1989), 1988.

\bibitem{drezet_narasimhan}
J.-M. Drezet and M.~S. Narasimhan.
\newblock Groupe de {P}icard des vari\'et\'es de modules de fibr\'es
  semi-stables sur les courbes alg\'ebriques.
\newblock {\em Invent. Math.}, 97(1):53--94, 1989.

\bibitem{vangeemen_izadi}
{B. van} Geemen and E.~Izadi.
\newblock The tangent space to the moduli space of vector bundles on a curve
  and the singular locus of the theta divisor of the {J}acobian.
\newblock {\em J. Algebraic Geom.}, 10(1):133--177, 2001.

\bibitem{kumar}
C.~Kumar.
\newblock Invariant vector bundles of rank 2 on hyperelliptic curves.
\newblock {\em Michigan Math. J.}, 47(3):575--584, 2000.

\bibitem{kumar_linearsystems}
C.~Kumar.
\newblock Linear systems and quotients of projective space.
\newblock {\em Bull. London Math. Soc.}, 35(2):152--160, 2003.

\bibitem{lange_narasimhan}
H.~Lange and M.~S. Narasimhan.
\newblock Maximal subbundles of rank two vector bundles on curves.
\newblock {\em Math. Ann.}, 266(1):55--72, 1983.

\bibitem{massarenti_rischter}
A.~Massarenti and R.~Rischter.
\newblock Non-secant defectivity via osculating projections, 2016.

\bibitem{narasimhan_ramanan_moduli}
M.~S. Narasimhan and S.~Ramanan.
\newblock Moduli of vector bundles on a compact {R}iemann surface.
\newblock {\em Ann. of Math. (2)}, 89:14--51, 1969.

\bibitem{newstead_stable}
P.~E. Newstead.
\newblock Stable bundles of rank {$2$} and odd degree over a curve of genus
  {$2$}.
\newblock {\em Topology}, 7:205--215, 1968.

\bibitem{ortega_coble}
A.~Ortega.
\newblock On the moduli space of rank 3 vector bundles on a genus 2 curve and
  the {C}oble cubic.
\newblock {\em J. Algebraic Geom.}, 14(2):327--356, 2005.

\bibitem{pareschi_popa}
G.~Pareschi and M.~Popa.
\newblock Regularity on abelian varieties. {I}.
\newblock {\em J. Amer. Math. Soc.}, 16(2):285--302, 2003.

\bibitem{pauly_coble}
C.~Pauly.
\newblock Self-duality of {C}oble's quartic hypersurface and applications.
\newblock {\em Michigan Math. J.}, 50(3):551--574, 2002.

\bibitem{raynaud}
M.~Raynaud.
\newblock Sections des fibr\'es vectoriels sur une courbe.
\newblock {\em Bull. Soc. Math. France}, 110(1):103--125, 1982.

\end{thebibliography}
\bibliographystyle{plain}
  
\end{document}